\documentclass[reqno]{amsart}
\usepackage{graphicx}
\usepackage[T1]{fontenc}

\usepackage{amsmath,amsthm,amssymb,amsfonts}
\newtheorem{theo}{Theorem}
\newtheorem{prop}{Proposition}

\newtheorem{lem}{Lemma}

\theoremstyle{definition}
\newtheorem{defi}{Definition}
\newtheorem{ex}{Example}
\newtheorem{remarque}{Remark}

\newcommand{\RR}{\mathbb{R}}
\newcommand{\ZZ}{\mathbb{Z}}

\begin{document}

\title{The systolic constant of orientable Bieberbach 3-manifolds}

\author{Chady El Mir}

\author{Jacques Lafontaine}

\date{\today}

\maketitle

\begin{center}
Department of Mathematics and computer science\\
Beirut Arab University \\
P.O.Box 11 - 50 - 20 Riad El Solh 11072809\\
              Beirut, Lebanon\\ 
email: {c.elmir@bau.edu.lb}\\

\&\\

              Institut de Math\'ematiques et Mod\'elisation de Montpellier \\
CNRS, UMR 5149\\
Universit\'e Montpellier 2\\
              CC 0051, Place Eug\`ene Bataillon\\
              F-34095 Montpellier Cedex 5, France\\
          
email: {jaclaf@math.univ-montp2.fr}

\end{center}

\begin{abstract}
A compact manifold is called \emph{Bieberbach} if it carries a flat Riemannian metric. Bieberbach manifolds are aspherical, therefore the supremum of their systolic ratio, over the set of Riemannian metrics, is finite by a fundamental result of M.~Gromov.

We study the optimal systolic ratio of compact of $3$-dimensional orientable Bieberbach manifolds which are not tori,
and prove that it cannot be realized by a flat metric. We also highlight a metric that we construct on one type of such manifolds ($C_2$)
which has interesting geometric properties :
it is extremal in its conformal class and the systole is realized by ``very many'' geodesics.\\

{\sc R\'esum\'e}. Une vari\'et\'e compacte est appel\'ee de Bieberbach si elle porte une m\'etrique riemannienne plate. Les vari\'et\'es de Bieberbach sont asph\'eriques, par cons\'equent le supremum de leur quotient systolique, sur l'ensemble des m\'etriques riemanniennes, est fini d'apr\`es un r\'esultat fondamental de M.~Gromov.

On \'etudie le quotient systolique optimal des $3$-vari\'et\'es de Bieberbach compactes et orientables qui ne sont pas des tores, et on d\'emontre qu'il n'est pas r\'ealis\'e par une m\'etrique plate. De plus, on met en \'evidence une m\'etrique que l'on construit sur un type de telles vari\'et\'es ($C_2$)
qui a une g\'eom\'etrie int\'eressante~: elle est extr\^emale dans sa classe conforme et poss\`ede
de ``nombreuses'' g\'eod\'esiques systoliques.

\vspace{0.5cm}

\vspace{0.5cm}

\end{abstract}

\section{Introduction}

\subsection{Motivations and main result}

The systole of a compact non simply connected Riemannian manifold $(M^n,g)$ is the shortest length of a non contractible closed curve, we denote it by $\mathrm{Sys}(g)$. We are interested in the \emph{systolic ratio}$\frac{\mathrm{Sys}(g)^n}{\mathrm{Vol}(g)}$, which is a scale-invariant quantity.
  Note that it is well defined even if $g$ is only continuous, i.e. a continuous section of the fiber bundle $S^2T^\ast M$ of symmetric forms.

An isosystolic inequality on a manifold $M$ is a inequality of the form

$$\frac{\mathrm{Sys}(g)^n}{\mathrm{Vol}(g)} \leq C  <+\infty$$
%%%

that holds for any Riemannian metric $g$ on $M$. The smallest such constant $C$ is called the \emph{systolic constant}.

A \emph{systolic geodesic} is a closed curve, not homotopically trivial, whose length is equal to the systole.

A precursory work about systolic inequalities is that of Charles Loewner (1949, unpublished) and his student P. M. Pu (\cite{pu}).
The fundamental article of Misha Gromov (cf. \cite{gromov} ) was the birth certificate
of \emph{systolic geometry}, that is the study of systolic inequalities on manifolds.
For more details about the subject see the survey  \cite{bki} of Marcel Berger and the book \cite{katz} of Mikha\"il Katz.

As far as this work is concerned, we put emphasis on two facts :

a)  Compact manifolds which are covered by $\RR^n$  do satisfy an isosystolic inequality (this is just
an application of the general result of  \cite{gromov}.

b) The systolic constant is in general fairly non explicit (see \cite{gromov} again).

In fact, the supremum of $\frac{\mathrm{sys}(g)^n}{\mathrm{vol}(g)}$ is known in three cases only:
\begin{enumerate}
 \item  The torus $T^2$: the supremum is achieved for the flat hexagonal metric (C. Loewner, 1949, see \cite{ghl}, p.95)
\item The real projective plane : the supremum is achieved for the constant curvature metric (P.M. Pu 1952, see \cite{pu}). 

\item The Klein bottle : the supremum is achieved for a singular Riemannian metric,
with constant curvature $+1$ outside a singular curve (C. Bavard, see \cite{bavard}).
\end{enumerate}

However, if the systole is replaced by the so-called  \emph{stable systole} $\mathrm{stsys}(g)$, the explicit inequality
$$\frac{\mathrm{stsys}(g)^n}{\mathrm{Vol}(g)} \leq (\gamma_n)^{n/2}$$
(where  $\gamma_n$ 
denotes the Hermite constant) %(see \cite{katz}, p.41 for this notion)
holds for the $n$-dimensional torus,
and the bound is attained for some flat tori. % $\mathbf{R}^n/\Lambda$.
See \cite{katz}, ch. 17, for a more general result (Burago-Ivanov-Gromov inequality), and details
about the stable systole.

In the present work, we are interested in \emph{Bieberbach manifolds}, i.e.
compact manifolds that carry a flat Riemannian metric. Standard Riemannian geometry shows that these manifolds are covered 
by $\RR^n$, and therefore satisfy an isosystolic inequality.  

For $n=3$, there are six orientable Beberbach manifolds (Hantzsche--Wendt classification, see  \cite{wolf} or \ref{classi} below).
One of them is the torus  $T^3$ for which, in contrast with the stable case, it is not known whether the optimal metric  is flat.
As for the other manifolds, since their first Betti number is strictly smaller than $3$,
the assumptions of Burago-Ivanov-Gromov inequality are not satisfied.

Our result is the following :

\begin{theo}
Let $M$ be a Bieberbach orientable manifold of dimension $3$ that is not a torus.
Then there exists on $M$ a Riemannian metric $g$ such that,
for any flat metric $h$,
$$
\frac{(\mathrm{sys}(h))^3}{\mathrm{vol}(h)} <   \frac{(\mathrm{sys}(g))^3}{\mathrm{vol}(g)}
$$
\end{theo}
(See theorem \ref{c2sing} for an example of a more precise statement.)

We already proved the same result for non-orientable $3$-dimensional Bieberbach manifolds (see \cite{E-L}). The main idea consisted in the fact that we can get any such manifold by suspending a  Klein bottle. Then, by taking a suitable metric suspension of the Bavard metric, one can beat the systolic ratio of flat metrics.
The orientable case is a bit more involved.

Actually this technique does not work for the orientable case, since the manifolds cannot be foliated by Klein bottles or M\"obius bands as in the non-orientable case. The key idea is that all these manifolds contain isolated systolic geodesics in the case of the extremal metric among the flat ones. Therefore we will be able to decrease the volume by "digging"  around the geodesic systolic  \emph{without shortening the systole} (multiplying the metric in the direction of this systolic geodesic by a function equal to 1 on the geodesic and less than 1 in its neighborhood).

The case of the manifold $C_{2,2}$ (see \ref{classi} below) is straightforward thanks to the following (folk) result.

\emph{If $g$ is an extremal Riemannian metric (possibly singular) on $M$, the systolic geodesics do cover $M$} (see \cite{calabi}).\\

This property is satisfied by flat tori, and real projectives endowed with their metrics of constant curvature.
On Bieberbach manifolds of dimension $3$, the metrics that optimize the systolic ratio \emph{among flat metrics}
also satisfy this property, except for the manifold $C_{2,2}$ (cf. \cite{wolf} p.117-118, and the suggestive figure of \cite{th}, p.236). This  property gives the result for $C_{2,2}$.%%

The metrics that we construct also satisfy this property, and so there is no obvious obstacle that prevents them from being extremal. Actually we highlight the metric we construct on the manifold $C_2$, (see sections \ref{c2ext} and \ref{torusC2}) because it verifies 
an even stronger property : \emph{it is covered by systolic geodesics of any systolic class}
(a systolic class is a free homotopy class which  contains at least one systolic geodesic).

It is worth-noting that the optimal flat metric on $C_2$ \emph{does not} satisfy this property.
 We also prove in section \ref{c2ext} that this metric is extremal in its conformal class.

\section{Flat manifolds}

\subsection{Classification of flat manifolds}
Compact flat manifolds are quotients $\mathbb{R}^n/\Gamma$, where $\Gamma$ is a discrete cocompact subgroup of affine isometries of $\RR^n$ acting freely. By the theorem of Bieberbach $\Gamma$ is an extension of a finite group $G$ by a lattice $\Lambda$ of $\RR^n$ (see for example \cite{charlap}). This lattice is the subgroup of the elements of $\Gamma$ that are translations, we obtain then the following exact sequence:

$$
0\longrightarrow\Lambda\longrightarrow\Gamma\longrightarrow G\longrightarrow 1
$$

Actually, if $M$ is a flat manifold, $M$ is the quotient of the flat torus $\mathbb{R}^n/\Lambda$ by an isometry group isomorphic to $G$.
Two compact and flat manifolds $\RR^n/\Gamma$ and $\RR^n/\Gamma^\prime$ are homeomorphic if and only if the groups $\Gamma$ and $\Gamma^\prime$ are isomorphic. Such groups are then conjugate by an affine isomorphism of $\RR^n$: two compact and flat homeomorphic manifolds are affinely diffeomorphic.

\subsection{3-dimensional orientable flat manifolds} \label{classi}

 The classification of flat manifolds of dimension $3$ results of a direct method of classification of discrete, cocompact subgroups of $\mathrm{Isom}(\RR^3)$ operating freely. 
This classification is due to W. Hantzsche and H. Wendt  (1935), and exposed in the book \cite{wolf} of J.A. Wolf.
There exist ten compact and flat manifolds of dimension $3$ up to diffeomorphism, six are orientable and four are not.

In the orientable case, the types are caracterized by the holonomy group $G$, this reason motivates
the notation of W. Thurston (cf. \cite{th} p.235) that we follow.
 A rotation of angle $\alpha$ around an axis $a$ will be denoted by $r_{a,\alpha}$.\\

i) $G=\{1\}$: type $C_1$. This is the torus, it is the quotient of $\RR^3$ by an arbitrary lattice of $\RR^3$.

ii)  $G=\ZZ_2$: type  $C_2$. Given a basis $(a_1,a_2,a_3)$ of $\RR^3$ with $a_3 \perp (a_1,a_2)$, let $\Gamma$ be the subgroup of isometries of $\RR^3$ generated by $t_{a_3/2} \circ r_{a_3,\pi}$ and the translations $ t_{a_1}$ and $t_{a_2}$. The quotient $\RR^3/\Gamma$ is a manifold of type $C_2$. 
The lattice $\Lambda$ generated by $t_{a_1}$, $t_{a_2}$ et $t_{a_3}$ is of index $2$ in $\Gamma$.\\

This manifold is the quotient of the torus $\RR^3 /\Lambda$ by the cyclic group of order $2$ generated by (the quotient map of) $t_{a_3/2} \circ r_{a_3,\pi}$.\\

iii) $G=\ZZ_4$: type $C_4$. Given an orthogonal basis $(a_1,a_2,a_3)$ of $\RR^3$ with $|a_1|=|a_2|$, let $\Gamma$ be the subgroup of isometries of $\RR^3$ generated by $t_{a_3/4} \circ r_{a_3,\pi/2}$ and the translations $ t_{a_1}$ et $t_{a_2}$. The quotient $\RR^3/\Gamma$ is a manifold of type $C_4$. The lattice $\Lambda$ generated by $t_{a_1}$,$t_{a_2}$ and $t_{a_3}$ is of index $4$ in $\Gamma$.\\

This manifold is the quotient of the torus $\RR^3 /\Lambda$ by the cyclic group of order $4$ generated by (the image of) $t_{a_3/4} \circ r_{a_3,\pi/2}$. It is also the quotient of $C_2$ (the basis $(a_1,a_2,a_3)$ should be chosen orthogonal with $|a_1|=|a_2|$), by the subgroup generated by $t_{a_3/4} \circ r_{a_3,\pi/2}$.

iv)  $G=\ZZ_6$: type $ C_6$. Given a basis $(a_1,a_2,a_3)$ of $\RR^3$ with $a_3 \perp (a_1,a_2)$, $|a_1|=|a_2|$ and $(a_1,a_2)= \pi/3$, let $\Gamma$ be the subgroup of isometries of $\RR^3$ generated by $t_{a_3/6} \circ r_{a_3,\pi/3}$ and the translations $ t_{a_1}$ et $t_{a_2}$. The quotient $\RR^3/\Gamma$ is a manifold of type $C_6$. The lattice $\Lambda$ generated by $t_{a_1}$,$t_{a_2}$ and $t_{a_3}$ is of index $6$ in $\Gamma$.\\

The manifold $C_6$ is the quotient of the torus $\RR^3 /\Lambda$ by the cyclic group of order $6$ generated by (the image of) $t_{a_3/6} \circ r_{a_3,\pi/3}$. It is also the quotient of $C_2$ by the subgroup generated by $t_{a_3/6} \circ r_{a_1,\pi/3}$.\\

v) $G=\ZZ_3$: type $C_3$. Given a basis $(a_1,a_2,a_3)$ of $\RR^3$ with $a_3 \perp (a_1,a_2)$, $|a_1|=|a_2|$ and $(a_1,a_2)= 2\pi/3$, let $\Gamma$ be the subgroup of isometries of $\RR^3$ generated by $t_{a_3/3} \circ r_{a_3,2\pi/3}$ and the translations $ t_{a_1}$ et $t_{a_2}$. The quotient $\RR^3/\Gamma$ is a manifold of type $C_3$.  The lattice $\Lambda$ generated by $t_{a_1}$,$t_{a_2}$ and $t_{a_3}$ is of index $3$ in $\Gamma$. This manifold is the quotient of the torus $\RR^3 /\Lambda$ but it is not a quotient of $C_2$.\\

vi) $G=\ZZ_2 \times \ZZ_2$: type $C_{2,2}$. Given an orthogonal basis $(a_1,a_2,a_3)$ of $\RR^3$, let $\Gamma$ be the subgroup of isometries of $\RR^3$ generated by $t_{a_1/2} \circ r_{a_1,\pi}$, $t_{(\frac{a_1+a_2}{2})} \circ r_{a_2, \pi}$ and $t_{(\frac{a_1+a_2+a_3}{2})} \circ r_{a_3, \pi}$. The quotient $\RR^3/ \Gamma$ is the manifold $C_{2,2}$. This time, the  holonomy group $G$ is not cyclic, it is equal isomorphic to $\ZZ_2 \times \ZZ_2$.

\section{Singular metrics on Bieberbach manifolds} \label{sing}

All the singular metrics we will use give rise to \emph{length spaces} (cf. \cite{bbi}).
 In this section  we will define these metrics in the setting of Riemannian polyhedrons. For further details on this notion see \cite{babenko2}.\\
A polyhedron is a topological space endowed with a triangulation, i.e.  divided into simplexes glued together by their faces. We denote  by $\Sigma$ an arbitrary simplexe of a polyhedron $P$.\\

\begin{defi} 
A Riemannian metric on a polyhedron $P$ is a family of Riemannian metrics $\{g_\Sigma\}_{\Sigma \in I}$, where $I$ is in bijection with the set of simplexes of $P$. 
These metrics should satisfy the following conditions: \begin{enumerate}

\item Every $g_\Sigma$ is a smooth metric in the interior of the simplex $\Sigma$.

\item The metrics $g_\Sigma$ co\"\i ncide on the  faces; i.e. for any pair of simplexes $\Sigma_1$, $\Sigma_2$, we have the equality
$$g_{\Sigma_1}|_{\Sigma_1 \bigcap \Sigma_2}=g_{\Sigma_2}|_{\Sigma_1 \bigcap \Sigma_2}$$
\end{enumerate}
\end{defi}

Such a Riemannian structure on the polyhedron  allows us to calculate the length of any piecewise smooth curve in $P$, this way, the polyhedron $(P,g)$ turns out to be  a length space.

Note that the Riemannian measure is defined exactly as in the smooth case. The volume of the singular part is equal to zero.\\ 

The geodesics of a Riemannian polyhedron are the geodesics of the associated length structure (see \cite{bbi}). 
In the interior of a simplex $(\Sigma, g_\Sigma)$, the first variation formula shows that such a geodesic is a geodesic of $g_\Sigma$ in the Riemannian sense (see \cite{E-L} p. 102 for an example). %%

\subsection{The Klein-Bavard bottle}

We know by Bavard (\cite{bavard}) that the unique extremal Klein bottle for the isosystolic inequality is not flat (see also \cite{gromov}), it is singular and has curvature equal to  $+1$ outside the singularities:

We start with the round sphere, and we locate the points by their latitude $\phi$ and their longitude $\theta$. For $\phi_o\in ]0,\pi/2[$, let $\Sigma_{\phi_o}$ be the domain defined by  $\vert\phi\vert \le \phi_o$. In $\Sigma_{\phi_o}$, the round metric is given by $d\phi^2 +\cos^2\phi d\theta^2$. Note that the universal cover of $\Sigma_{\phi_o}$ is the strip $\RR \times [-\phi_o,\phi_o]$ with the same metric. Here we introduce in $\RR^2$ the singular Riemannian metric

\begin{equation}
d\phi^2 +f^2(\phi)d\theta^2,
\label{sph}
\end{equation}

where $f$ is the $2\phi_0$ periodic function which agrees with $cos\phi$ in the interval $[-\phi_o, \phi_o]$.

\begin{ex} The metric on the Klein bottle that gives the best systolic ratio ($\frac{\pi}{2 \sqrt{2}}$) is obtained for $\phi_o=\frac{\pi}{4}$ by taking the quotient of the plane endowed with the metric  \ref{sph}
by the action of the group generated by

$$ (\theta,\phi)\mapsto (\theta+\pi,-\phi)\quad\hbox{et}
\quad  (\theta,\phi)\mapsto (\theta,\phi+ 4\phi_0).$$
\end{ex}

For more details on the Klein-Bavard bottle see \cite{E-L}, \cite{bavard2} and \cite{bavard4}.

\begin{remarque}

It may seem more natural to take the quotient of the plane (endowed with the metric  \ref{sph}) by the group generated by

$$ (\theta,\phi)\mapsto (\theta+\pi,-\phi)\quad\hbox{and}
\quad  (\theta,\phi)\mapsto (\theta,\phi+ 2\phi_0)$$

the surface we obtain is indeed a Klein bottle but it does not give the best systolic ratio: the geodesics closed by the correspondence $(\theta,\phi)\mapsto (\theta,\phi+ 2\phi_0)$ have a length equal to $\pi/2$ whereas the ones closed by the correspondence $(\theta,\phi)\mapsto (\theta,\phi+ 2\phi_0)$ have length $\pi$. It is then possible to reduce the volume without shortening the systole by reducing the metric in the direction of the long closed curves.

\end{remarque}

\subsection{Singular metrics on orientable Bieberbach manifolds}

\label{r3h}

Starting with an arbitrary lattice $\Delta $ of $\RR^2$, we introduce the associated Dirichlet-Vorono\"i  tiling. It is a tiling of the plane by hexagons (or rectangles if the lattice $\Delta$ is rectangle) $A_p$ centered at the  points $p$ of the lattice. Then a lattice of $\RR^3$ of the form $\Delta \times c \ZZ$, where $c >0$, allows us to tile $\RR^3$  naturally with hexagonal or rectangular prisms that we denote by $D_p$.\\

Now we introduce on $\RR^3$ the Riemannian singular metric $h(m)= dx^2 + dy^2 + \psi(m) dz^2$, 
where we set, for $m(x,y,z) \in D_p$,  $\psi(m)= \cos ^2 \mathrm{dist}\big((x,y),p\big)$, with $\mathrm{dist}\big((x,y),p\big)< \pi/2$. If $m $ is in two domains $D_p$ and $D_{p'}$ then $p$ and $p'$ are at the same distance from $m$: 
 the map $\psi$ is well defined. It is continuous,
but it is not $C^1$.  The connected component of the identity in the group of isometries of $(\RR^3,h)$ consists of the vertical translations $(x,y,z)\mapsto (x,y,z+c')$.
The translations by the vectors of $\Delta$ are also isometries. It is important to note that the metric $h$ can also be written in the form $dx^2+dy^2 +\cos ^2\big(d((x,y),\Delta)\big)$, where $d((x,y),\Delta)$ is the  distance from the point $(x,y)$ to the lattice $\Delta$. \\

The quotient of $(\RR^3,h)$ by the group $\Delta\times c\ZZ$ is a singular torus of dimension $3$.
We denote by $(T,h)$ this special torus.
The sections of $(T,h)$ by the planes $z=constant$ are $2-$dimensional totally geodesic flat tori. All these flat tori are isometric to $\RR^2/ \Delta$. Note that the map from $(T,h)$ onto the torus $\RR^2 / \Delta$, which consists in projecting onto the torus $z=0$, is a Riemannian submersion.\\

With a good choice of the lattice $\Delta$, the transformations $t_{a_3/n} \circ r_{a_3,2\pi/n}$ $(n=2,3,4,6)$, described in the classification of the flat orientable manifolds, become isometries of $(T,h)$ (The lattice $\Delta$ should be square to get $C_4$ and hexagonal to get $C_3$ and $C_6$). This way we get a family of singular Riemannian metrics on the manifolds of type $C_2$, $C_3$, $C_4$ and $C_6$.\\

\begin{remarque}
Actually this construction also works if we take the quotient of $(\RR^3,h)$ by the lattice $n \Delta \times c\ZZ$ ($n \in \mathbb{N}^*$). Starting with this torus, we can re-obtain all the manifolds $C_i$ $(i=2,3,4,6)$ exactly the same way as for the torus $(T,h)$. We will see in the next sections that taking $n=2$ is more useful to get "good systolic ratios" on these manifolds.

\end{remarque}

\section{Two singular tori and their systole} \label{toresing}

We take the quotient of the Riemannian singular space $(\RR^3,h)$ seen in section \ref{r3h}
by the lattice $2\Delta\times 2 \pi\ZZ$. We get a  $3$-dimensional torus $(T,g)$ whose singular locus is connected.
%nouveau 
It can be described as follows: the tiling of $\RR^3$ defined
by the Vorono\"i domains associated to $\Delta$ gives us
a triangulation (with four hexagonal faces in the generic case, four rectangular faces if the
lattice is rectangular) of $\RR^2/2\Delta$. Let $L$ be the 1-skeleton of
this triangulation. The singular locus is $L\times \RR/2\pi \ZZ$.

%It consists in the  boundaries of four hexagonal (or rectangular) prisms constituting a fundamental %domain for the action of $2\Delta\times 2\pi \ZZ$. 
The sections of $(T,g)$ by planes containing the axis of a domain $D_p$ are surfaces of curvature $+1$ as long as we stay in the interior of $ D_p$. These surfaces, that we denote by $S_p$, are actually spherical cylinders with boundary.

\begin{remarque} To preserve the systole and reduce the volume of the manifolds of type $C_i$ we have to take the quotient of $(\RR^3,h)$ by the lattice $2 \Delta \times 2 \pi \ZZ$ and not by $\Delta \times 2 \pi \ZZ$. This prevents shortening closed curves at the level of the surfaces $S_p$.

\end{remarque}

\subsection{The torus $(T,g_c)$}\label{squaretorus}
First suppose that the lattice $\Delta$ is square and generated by two vectors of norm $2a >0 $. This lattice is generated by three translations $t_1:(x,y,z) \longrightarrow (x+4a,y,z) $, $t_2:(x,y,z) \longrightarrow (x,y+4a,z) $ and  $t_3:(x,y,z) \longrightarrow (x,y,z + 2 \pi)$. We denote by $(T,g_c)$ the quotient torus.\\

Note that the symmetries with respect to the surfaces $x=pa$ and $y=qa$ where $p,q \in \ZZ$, are isometries of $(T,g_c)$.

\begin{lem} \label{sysgc}
The systole of $(T,g_c)$ is equal to $\inf \{4a,2\pi \cos (a \sqrt{2})\}$.
\end{lem}
\begin{proof}

Let $\gamma$ be a curve in $(\RR^3,h)$, from $m(x_0,y_0,z_0)$ to $t_1(m)$, then
$$l(\gamma) \geq  \int{\sqrt{x'^2+y'^2+\psi (x,y)z'^2}dt} \geq \int{x'dt} \geq 4a$$
Now, if $\gamma$ is a curve from $m(x_0,y_0,z_0)$ to $t_2(m)$ we find the same way as before that  $l(\gamma) \geq 4a$.
Finally, for a curve $\gamma$ from $m$ to $t_3(m)$, we have
$$l(\gamma)= \int{\sqrt{x'^2+y'^2+\psi (x,y) z'^2}dt} \geq \int{\inf (\psi) z'dt}= 2\pi \cos a \sqrt{2}$$
the equality is obtained for the points of the edges of the square prism $D_p$.
Using the same technique we can prove that the distance between a point $m$ and its image by the composition of several translations is greater or equal to $\inf \{4a,2\pi \cos a \sqrt{2}\}$.

\end{proof}
% nouveau
\subsection{The torus $(T,g_{hex}) $}\label{hextorus}

Suppose now that the lattice $\Delta$ is hexagonal and generated by two vectors of norm $2a >0$. 
The lattice $2 \Delta \times 2 \pi \ZZ$ is generated by the translations $T_1:(x,y,z) \longrightarrow (x+4a,y,z) $, $T_2: (x,y,z) \longrightarrow (x+2a,y +2a \sqrt{3}, z)$ and $T_3: (x,y,z) \longrightarrow (x, y, z+2 \pi)$. The manifold we get is a singular torus that we denote by $(T,g_{hex})$.

\begin{remarque} \label{isomghex} The symetries with respect to the surfaces $x=pa$, $y+\frac{x}{\sqrt{3}}= \frac{2pa}{\sqrt{3}}$ and $y-\frac{x}{\sqrt{3}}= \frac{2pa}{\sqrt{3}}$, give in the quotient isometries of $(T,g_{hex})$.

\end{remarque}

\begin{lem} \label{sysghex}
The systole of $(T,g_{hex})$ is equal to $\inf \{4a, 2 \pi \cos (2a/ \sqrt{3}) \}$.
\end{lem}

\begin{proof}

For any curve $\gamma$ from $m$ to $T_1(m)$ we have
$$l(\gamma) \geq  \int{\sqrt{x'^2+y'^2+\psi (x,y)z'^2}dt} \geq \int{x'dt} \geq 4a$$
the same inequality holds for any curve from $m$ to $T_2(m)$ since the situation is invariant by the rotation $r_{a_3,\pi/3}$ of angle $\pi/3$ around the axis $z$.\\
Finally, for any curve $\gamma$ from $m$ to $T_3(m)$, we have
$$l(\gamma)= \int{\sqrt{x'^2+y'^2+\psi (x,y) z'^2}dt} \geq \int{\inf (\psi) z'dt}= 2\pi \cos (2a \sqrt{3})$$
the equality is achieved for the points of the edges of the hexagonal prisms $D_p$.
The distance between a point $m$ and its  image by the composition of several translations is greater or equal to $\inf \{4a,2\pi \cos (2a \sqrt{3})\}$.

\end{proof}

\section{The systolic ratio of $C_2$} \label{C2}

\subsection{The systolic ratio of flat metrics}
The volume is equal to $\frac{1}{2} \det(a_2,a_1) |a_3|$ and the  systole is equal to $\inf\{|a_3|/2,s\}$, where $s$ is the systole of the flat torus of dimension $2$ defined by the lattice of basis $a_1,a_2$. We normalize such that $\frac{1}{2}|a_3|=1$, then the systolic ratio is equal to

$$
\frac{s^3}{\det(a_1,a_2)}\quad\hbox{if $s\le 1$ and}\quad \frac{1}{\det(a_1,a_2)}
\quad\hbox{if $s\ge 1$,}
$$
 
Since we have $\frac{s^2}{\det(a_1,a_2)} \leq \frac{2}{\sqrt{3}}$ (lattice of dimension $2$), the systolic ratio is less or equal to $2/\sqrt{3}$.

\subsection{A singular metric on  $C_2$ better than the flat ones} \label{totgeod}

We start by giving the following useful observations:

\begin{itemize}

\item The metric $h$ defined in section \ref{r3h} can be written locally (in the domain D) in cylindrical coordinates in the form $h=dr^2+r^2d\theta ^2 +\cos ^2r dz^2$ ($r$ is the distance to the vertical axes going through the center $p$ of the prism $D_p$, and $\theta$ is the angle  with respect to the axis $"x"$). \\

\item In restriction to a prism $D_p$, a surface $S_p$ of equation $\theta = \theta_0$ is totally geodesic. To see this, just remark that the length of any curve  $\gamma$ in $D_p$ joining two points of $\theta = \theta_0$ is always greater than its projection on this surface. This is simply due to the expression of the metric in the "cylindrical" coordinates: 
$$l(\gamma)= \int{\sqrt{{r'}^2+r^2 {\theta'}^2+ \cos^2r {z'}^2}dt} \geq \int{\sqrt{{r'}^2+ \cos^2r {z'}^2}}$$

\end{itemize}

\begin{lem} \label{miniproj}
Let $\gamma$ be a curve of the universal Riemannian covering of $(T,g_{hex})$ and $\gamma'$ its minimal projection on a hexagonal prism $D_p$, then we have $l(\gamma) \geq l(\gamma')$. (The \emph{minimal projection} of a point $m$ is here the point of $D_p$ 
at a minimal distance (Euclidean) of $m$. It is unique since $D_p$ is convex).
\end{lem}

\begin{proof}

If the minimal orthogonal projection is completely inside the singularity $x=pa$, i.e. if $\gamma'$ is in such a hypersurface, then

$$l(\gamma) =\int{\sqrt{x'^2+y'^2+\psi (x,y)z'^2}dt} \geq \int{\sqrt{y'^2+\psi (x,y)z'^2}dt}= l(\gamma')$$

but the situation is invariant by a rotation of angle $\pi/3$ around $p$; this shows that if $\gamma$ is projected on the surfaces $y+\frac{x}{\sqrt{3}}= \frac{2pa}{\sqrt{3}}$ or $y-\frac{x}{\sqrt{3}}= \frac{2pa}{\sqrt{3}}$ of the singularity, we have $l(\gamma) \geq l(\gamma')$. Finally, the result is true for any curve projected anywhere on the singularity.

\end{proof}

\begin{remarque} \label{parallel}
In fact the previous lemma holds even if we take the minimal projection on a hexagonal prism \emph{inside} $D_p$ and parallel to it. A prism is parallel to $D_p$ if every plane consisting its boundary is parallel to a plane of the boundary of $D_p$. This remark will be used in the improvement of the systolic ratio of the manifold $C_3$.

\end{remarque}

We denote by $\sigma$ the map $(x,y,z)\mapsto (-x,-y,z+\pi)$ (see type $C_2$ in section \ref{classi}),
and the correponding quotient maps.

\begin{remarque}

In the torus $(T,g_{hex})$ seen in section \ref{toresing}, the transformation $\sigma$ keeps $4$ geodesics globally invariant, these are the vertical axes containing the  $4$ centers of the (quotiented) prisms that constitute a fundamental domain of $C_2$.

\end{remarque}

\begin{lem} \label{distsigma}
For any point $m(r_0,\theta_0,z_0)$ in $(T,g_{hex})$ we have 
$$d_{(T,g_{hex})}(m,\sigma(m)) \geq \pi.$$ 
The equality is achieved for a geodesic of the surface $\theta=\theta_0$.
\end{lem}

\begin{proof}

Let $m(r_0,\theta_0,z_0)$ be a point in $D_p$, and $\gamma$ a curve in $(\RR^3,h)$ from $m$ to $\sigma(m)$. If $\gamma$ stays in $D_p$, then by Remark \ref{totgeod} we have $l(\gamma) \geq l(pr(\gamma))$ where $pr(\gamma)$ is the projection of $\gamma$ on the surface $\theta = \theta_0$. But $l(pr(\gamma)) \geq \pi$ since the metric on this surface is spherical ($dr^2 +cos^2r d\theta ^2$). 
Now if $\gamma$ gets out of the prism $D_p$, let $\gamma'$ be the curve obtained by taking the projection (minimal) of the part of $\gamma$ outside $D_p$ on the boundary $\partial D_p$, and by leaving the part inside $D_p$ unchanged. Then $\gamma'$ is a curve of $D_p$ from $m$ to $\sigma(m)$. Its length is greater or equal to $\pi$ (using the same argument of projection on the surface $\theta = \theta_0$). We conclude that $l(\gamma) \geq l(\gamma')$.

Then we have to calculate  in $(\RR^3,h)$ (a lower bound of) the distance to (a lift of) $\sigma(m)$ of the images of $m$ by translations. We denote by $\sigma_0$ any lift of $\sigma$ in $(\RR^3,h)$. If we translate $m$ by $T_3$, the situation will be equivalent to the one above since $T_3(m)$ and $\sigma_0(m)$ are conjugate by $\sigma_0^{-1}$. Now a curve $\gamma$ in $(\RR^3,h)$ from $\sigma_0(m)$ to $T_1(m)$ should go through at least $3$ domains $D_p$. Among these let $D'$ be the domain that neither contains $\sigma_0(m)$ nor $T_1(m)$. 
\begin{itemize}

\item If $\gamma$ stays in these three domains, let $\gamma'$ be the curve obtained by taking symmetrics of the parts of  $\gamma$ outside $D'$ with respect to the singular "plane" of $\partial D'$ beside the curve (see fig.3.1). The curve $\gamma'$ is in $D'$, it joins two conjugate points by the transformation $\sigma_0$, then $l(\gamma) \geq l(\gamma') \geq \pi$ (above argument). 

\item If $\gamma$ gets out of these domains, let $\gamma'$ be the curve obtained by projecting the part of $\gamma$ outside $D'$ on its boundary $\partial D'$. We get a continuous curve in $D'$ joining two conjugate points by $\sigma_0$, we conclude that $l(\gamma) \geq l(\gamma') \geq \pi$. 
\end{itemize}

Finally, note that the distance to $\sigma_0(m)$ of the composition of several translations of $m$ is too large by  arguments similar to those above.

\end{proof}

\begin{figure}
\begin{center}
\includegraphics[scale=1]{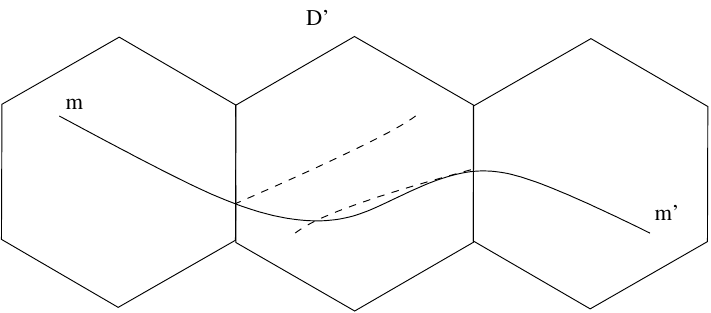}
\caption{A curve joining $m$ to $m'= T_1(\sigma(m))$ will go through $3$ domains $D_p$. For the parts of this curve outside $D'$ we take their symetrics with respect to the boundary $D'$.}
\end{center}
\end{figure}

\begin{remarque} \label{distgc}
The two preceding lemmas are also true for the torus $(T,g_c)$ and can be proved exactly the same way.
\end{remarque}

Summing up this discussion, we get the following result.

\begin{theo}
\label{c2sing}
Take the family of  (singular) metrics on $C_2$ defined as follows. 
\begin{itemize}
\item
Take a hexagonal lattice $\Delta$ in the plane.
\item
Equip the $3$-space with the metric
$$h_{x,y,z}=dx^2+dy^2+ \cos^2f(x,y,z)dz^2,\quad \textrm{with}\ f(x,y,z)=\mathrm{dist}\big((x,y),\Delta)\big)$$
\item
Take the quotient
by the group generated by $2\Delta$ and $\sigma:(x,y,z)\mapsto (-x,-y,z+\pi)$ 
\end{itemize}.

If the length of the shortest vector in $\Delta$ is equal to $\frac{\pi}{2}$ then 
$$
\frac{{\mathrm{Sys}}^3}{Vol}> \frac{2}{\sqrt 3}
$$
In other words : this metric has a bigger systolic ratio than the biggest systolic ratio of flat metrics on $C_2$.
\end{theo}

\begin{proof} This metric is just a quotient of the metric $(T,g_{hex})$
which was described in \ref{hextorus}. Denote it by $(C_2,g_{hex})$. It
depends on the real parameter $a$ ;  the minimal length of vectors in
the lattice $\Delta$ is $2a$.

The volume of $(C_2,g_{hex})$ is equal to $$\int_{0}^{\pi}  \iint_{D}  \cos \sqrt{x^2+y^2} dy dx dz $$ where $D$ is a regular hexagon of shortest distance between its opposite edges equal to $2a$.

The systole is equal to $$\inf \big \{Sys(T,g_{hex}),\inf \{dist_{(T,g_{hex})}(m,\sigma(m))\}\big \}$$ 
By Lemmas \ref{sysghex} and \ref{distsigma}, it is equal to $ \inf \{4a,2\pi\cos (a \sqrt{2}), \pi \} $. Then, for $a= \pi/4$, we have $Sys(C_2,g_{hex})= \pi$. Using the software "Maple" we find an approximation of the volume ($\simeq$ 2,80), then a simple calculation gives the systolic ratio $Sys^3(C_2,g_{hex})/Vol(C_2,g_{hex}) \simeq 1,38$.

\end{proof}

\subsection{The manifold $C_2$ as a quotient of the torus $(T,g_c)$}

To get a manifold homeomorphic to $C_2$, we can take an arbitrary plane lattice $\Delta$, then consider the quotient of $(\RR^3,h)$ by the same transformations as before. To increase the most the systolic ratio, $\Delta$ should have the smallest volume possible, i.e. it should be hexagonal. It is nevertheless interesting to get this manifold as a quotient of the torus $(T,g_c)$, i.e. when  $\Delta$ is the "special" square lattice. We denote by $(C_2,g_c)$ the quotient of $(T,g_c)$ by the subgroup generated by $\sigma$ ($(x,y,z)\mapsto (-x,-y,z+\pi)$).

When $a= \pi/4$, the intersection of $(T,g_c)$ with one of the planes $x=0$ or $y=0$ is the covering torus of the Klein-Bavard bottle (c.f. \cite{bavard}, see also \cite{E-L}). It turns out that with a good choice of the parameter $a$ the manifold $(C_2,g_c)$ admits a systolic ratio greater than $\sqrt{3}/2$, and the calculation is based, as in the case of $(T,g_{hex})$, on the fact that the distance in $(T,g_c)$ between a point and its image by $\sigma$ is greater than $\pi$ (c.f. \ref{distsigma}).

\begin{prop}
If the length $a$ of the shortest vector of $\Delta$ satisfies the equation $2a- \pi \cos a\sqrt{2}=0$, then the systolic ratio $\frac{Sys^3(C_2,g_c)}{Vol(C_2,g_c)}$ is greater than  $2/\sqrt{3} $. It is approximately equal to $1,18$.
\end{prop}

\section{The systolic ratio of $C_4$, $C_6$, $C_3$ and $C_{2,2}$} \label{Ci}

\subsection{Type $C_4$}

In the flat case, we saw that $C_4$ is the quotient of $C_2$ (the basis $(a_1,a_2,a_3)$ should be orthogonal with $|a_1|=|a_2|$) by the subgroup generated by $t_{a_3/4} \circ r_{a_3,\pi/2}$. Its volume is equal to $|a_1||a_2||a_3|/4$, and the  systole is equal to $\inf \{|a_1|,|a_2|,|a_3|/4\}$. The systolic ratio is less than $1$. \\

The transformation $\tau=t_{a_3/4} \circ r_{a_3,\pi/2}$ is also an isometry of $g_c$ and the quotient of $(C_2,g_c)$  by this isometry  gives a manifold of type $C_4$, we denote it by $(C_4,g_c)$.

\begin{remarque} 
The transformations $\tau$ and $\tau^{-1}$ are of order $4$ in $(T,g_c)$  and keep $2$ geodesics globally invariant.

The transformation $\tau^2$ is of order $2$ in $(T,g_c)$ and keeps, in addition to the geodesics fixed by the transformation $\tau$, $2$ others globally invariant. They are the vertical geodesics going through the points of the lattice $\Delta$ (see fig.2). 

\end{remarque}

\begin{figure}
\begin{center}
\includegraphics[scale=1.5]{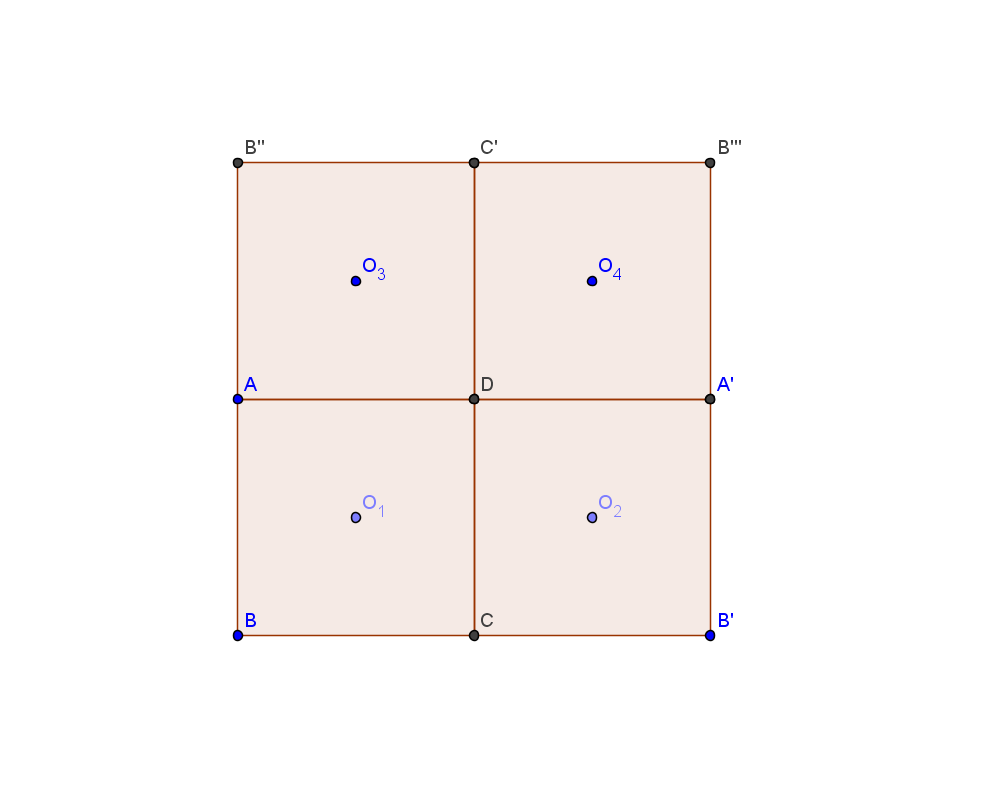}
\caption{The transformations $\tau$ et $\tau^{-1}$ keep fixed
the vertical axes going through the points $O_1$ et $O_4$.
The transformation $\tau^2$ keep fixed these same axes, as well as the vertical ones going through the points $O_3$ and $O_4$.}
\end{center}
\end{figure}

\begin{theo}\label{c4sing}

Consider the singular metric on $C_4$ given as follows. 
\begin{itemize}
\item
Take a square lattice $\Delta$ in the plane.
\item
Equip the $3$-space with the metric
$$h_{x,y,z}=dx^2+dy^2+ \cos^2f(x,y,z)dz^2,\quad \textrm{with}\ f(x,y,z)=\mathrm{dist}\big((x,y),\Delta)\big)$$
\item
Take the quotient
by the group generated by $2\Delta$ and $\tau:(x,y,z)\mapsto (-y,x,z+\frac{\pi}{2})$
\end{itemize}

If the length of the shortest vector of $\Delta$ is equal to $\frac{\pi}{4}$ then 
the systole is equal to $\pi/2$ and the systolic ratio is greater than $1$.

In other words : this metric has a bigger systolic ratio than the biggest systolic ratio of flat metrics on $C_4$.
\end{theo}

\begin{proof}

Denote by $(C_4,g_c)$ this metric. It depends on the real parameter $a$, where $2a$
is the lengnt of the shortest vector in $\Delta$.
The systole of $(T,g_c)$ is equal to $\inf \{4a,2\pi \cos (a \sqrt{2})\}$. By proposition \ref{distgc} we know that $d(m,\tau^2 (m)) \geq \pi$ ($\tau^2= \sigma$), therefore we are reduced to finding a "good" lower bound of $\tau$. Using the triangular inequality in $(T,g_c)$, we have 
$$d(m,\tau ^2(m)) \leq d(m,\tau(m)) + d(\tau (m),\tau ^2 (m)) $$
but $d(p,\tau (p)) = d(\tau(p),\tau ^2(p))$ since $\tau$ is an isometry of $(T,g_c)$. Then
$d(m,\tau (m)) \geq \pi/2$, and the equality is achieved for the points $m$ of the rotation axis. Note  that using the same method, we get a good lower bound of $\tau^3= \tau^{-1}$.\\
Finally for $a=\pi/8$ the systole of $(C_4,g_c)$ is equal to $\pi/2$. The volume is equal to $$ 4 \int_{0}^{\frac{\pi}{2}} \int_{-\frac{\pi}{8}}^{\frac{\pi}{8}} \int_{-\frac{\pi}{8}}^{\frac{\pi}{8}}  \cos \sqrt{x^2+y^2} dx dy dz $$
Using  Maple, we find the systolic ratio of our manifold, it is approximately equal to $1,05 > 1$.

\end{proof}

\subsection{Type $C_6$}

In the flat case, the volume is equal to $\frac{1}{6} det(a_1,a_2) |a_3|$ and the systole is equal to $\inf\{|a_3|/6,s\}$, where $s$ is the systole of the flat 2-dimensional torus defined by the lattice of basis $a_1,a_2$. Considering the usual normalisation $\frac{1}{6}|a_3|=1$, the systolic ratio is equal to

$$
\frac{s^3}{\det(a_1,a_2)}\quad\hbox{if $s\le 1$ and}\quad \frac{1}{\det(a_1,a_2)}
\quad\hbox{if $s\ge 1$,}
$$
 
It is less than $2/\sqrt{3}$.\\
% \'egalit\'e

Now to improve this systolic ratio, we start with the hexagonal torus $(T,g_{hex})$ defined in \ref{toresing}, since the lattice $\Delta$ should be hexagonal. To get the manifold $C_6$, we take the quotient of $(T,g_{hex})$ by the subgroup generated by the isometry $\phi=t_{a_3/6} \circ r_{a_3,\pi/3}$, the resulting manifold is $(C_6,g_{hex})$.

\begin{remarque}

The transformations $\phi$ and $\phi^{-1}$ are of order $6$ in $(T,g_{hex})$ and keep only one geodesic globally invariant.

The transformations $\phi^2$ and $\phi^4$ are of order $3$ in $(T,g_{hex})$ and keep, in addition to the one of $\phi$, $2$ vertical geodesics globally invariant.

The transformation $\phi^3$ is of order $2$ and keeps, in addition to the one kept invariant by $\phi$, $3$ vertical geodesics globally invariant (see fig.3).

\end{remarque}

\begin{figure}
\begin{center}
\includegraphics[scale=1]{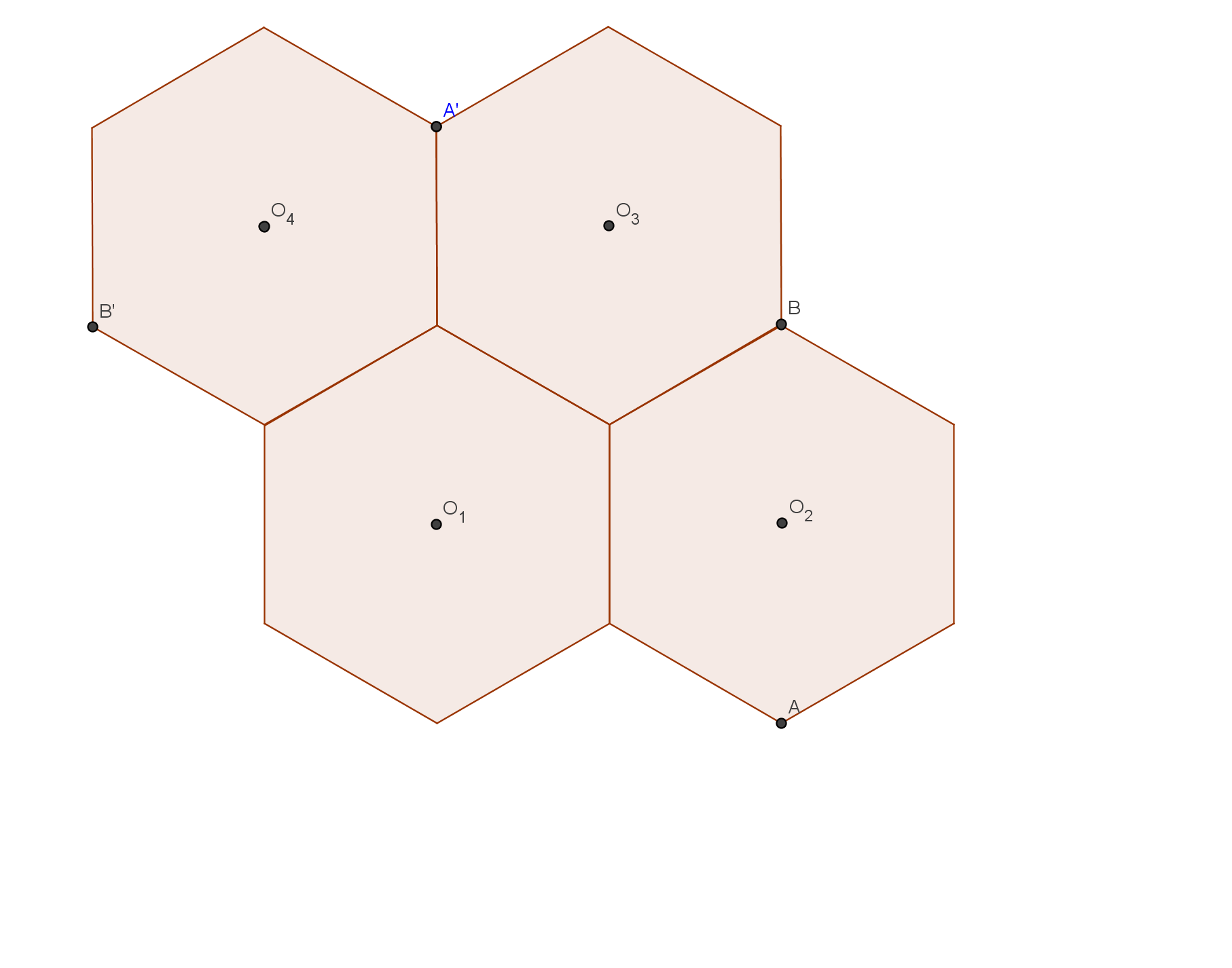}
\caption{The transformations $\phi$ and $\phi^{-1}$ only keep fixed the vertical axis going through the point $O_1$.
The transformations $\phi^2$ and $\phi^{-2}$ keep fixed, in addition to the axis going through $O_1$, the vertical axes
going through the points $A$ et $B$.
The transformation $\phi^3$ keeps fixed, in addition to these three axes, the vertical ones going through the points $O_i$, ($i=2,3,4$).}
\end{center}
\end{figure}

% nouveau 
This discussion gives the following result.

\begin{theo} If the norm of the shortest vector of $\Delta$ is $\pi/6$,
the quotient metric on $C_6$ with respect to the metric $(T,g_{hex})$ (i.e. the manifold $(C_6,g_{hex})$)
has a systolic ratio bigger than $2/\sqrt{3}$, namely bigger
than the best systolic ratio for flat metrics on $C_6$.

\end{theo}

\begin{proof}
The proof is exactly the same as for theorem \ref{c4sing}.
\end{proof}

\subsection{Type $C_{2,2}$}

In the flat case, the systole is equal to $\inf \{ a_1/2,a_2/2,a_3/2 \}$. The volume is equal to $\frac{|a_1||a_2||a_3|}{4}$. The systolic ratio is less than $1/2$, the equality is achieved if and only if $|a_1|=|a_2|=|a_3|$. In this case the systolic geodesics are isolated and therefore do not cover the manifold $C_{2,2}$. 

The criterion seen in the introduction allows us to conclude that the flat metrics on $C_{2,2}$ are not the best for the isosystolic inequality.

\subsection{Type $C_3$}

In the flat case,  
the volume is equal to $\frac{1}{3} \det(a_1,a_2) |a_3|$ but also to 
$\frac{\sqrt 3}{6}\vert a_1\vert \vert a_3\vert$,  and the systole is equal to $\inf\{|a_3|/3,\vert a_1\vert\}$, since the lattice generated by $a_1$ and
$a_2$ is hexagonal.
We conclude that the systolic ratio is less or equal to $2/\sqrt{3}$.
The equality is achieved for $|a_3|=3|a_2|=3|a_1|$.

To improve this systolic ratio, we start with the hexagonal torus $(T,g_{hex})$ defined in section \ref{toresing}.
To get the manifold $C_3$, we take the quotient of $(T,g_{hex})$ by the subgroup generated by the isometry $\varphi=t_{a_3/3} \circ r_{a_3,2\pi/3}$, the result is the manifold $(C_3,g_{hex})$.

Since the manifold $C_3$ is not a quotient of $C_2$, it does not contain surfaces that are Klein bottles or M\"obius bands, therefore our previous methods of getting a lower bound for the systole cannot be applied (it still satisfies the property of having isolated systolic geodesics ). Therefore, a special and more general argument is necessary.\\

Let $\varphi_c$ be the isometry of $(T,g_{hex})$ which sends $(p,z) $ to the point $(r_{2\pi/3}(p),z +  c)$. The quotient of $(T,g_{hex})$ by the subgroup generated by $\varphi _c$ is clearly a manifold homeomorphic to $C_3$, we denote it by $(C_3,g_{hex}^c)$ (here we suppose that the vertical translation in $(T,g_{hex})$ is $(p,z)\mapsto (p,z+2c)$). Let $\gamma$ be the vertical geodesic (closed by $\varphi_c$) in a domain $D_p$ in $(C_3,g_{hex}^c)$ going through the point $p$, it has length equal to $c$. Now let $H$ be a piecewise smooth variation of $\gamma$ through geodesics joining a point $m$ to  $\varphi_c(m)$, we impose that these curves stay in $D_p$ and do not touch the singularity. 

\begin{lem}
The second variation of $H$ at the curve $\gamma$ is strictly positive if $0<c<2\pi /3$.

\end{lem}

\begin{proof}

Let $O$ be a small tubular neighborhood of $\gamma$ and let $\Omega$ be the set of geodesics in $D_p$ from $m_t \in O$ to $\varphi_c(m_t)$ that do not touch the singularity (one parameter family since the situation is invariant under rotation around $\gamma$). Then

$$H:]-\epsilon,\epsilon[ \longrightarrow \Omega  $$ 

$$t \longrightarrow \gamma_t:[0,1] \rightarrow  (C_3,g_{hex}^c)$$

is such that $H(0)= \gamma$. Let $T=\frac{\partial \gamma_t}{\partial s}$ (velocity vector of $\gamma_t$), and $V=\frac{\partial \gamma}{\partial t}|_{\gamma_t}$ (the Jacobi field along $\gamma$), and set $L=\int_0^c{|T|ds}$. We have then
$$\frac{\partial L}{\partial t}=\frac{1}{L(t)} \int_0^c{g_{hex}^c(V,\nabla_V T) ds} $$ 

since $\nabla_T V-\nabla_V T=[V,T]=0$ we get

$$L \frac{\partial L}{\partial t} = [g_{hex}^c(V,T)]_0^c  \qquad \text{(1st variation formula)}$$

Now
$$ \frac{\partial}{\partial t}(L \frac{\partial L}{\partial t})= (\frac{\partial L}{\partial t})^2 +L \frac{\partial ^2 L}{\partial t^2}$$

$$= \int_0^c{(|\nabla _T V|^2 +g_{hex}^c(T,\nabla _V\nabla_T V))ds}$$

$$=\int_0^c{|\nabla _T V|^2} +\int_0^c {g_{hex}^c(T,\nabla _T\nabla_V V)}+ \int_0^c{g_{hex}^c(R(V,T)V,T)}$$ where $R$ is the curvature tensor of $g_{hex}^c$. Now since the curvature in the direction of the plane $(T,V)$ is equal to $1$ we get

$$L \frac{\partial ^2 L}{\partial t^2}=\int_0^c {|\nabla _T V|^2}-L^2\int_0^c{|V|^2}$$

$$=\int_0^c{(\nabla_T g_{hex}^c(V,\nabla_T V)-g_{hex}^c(V,R(T,V)T))}-L^2\int_0^c{|V|^2}= g_{hex}^c(V,\nabla_T V)|_0^c$$

(see \cite{cheeger} p. 20 for more details on the second variation formula).

Now $V$ is a Jacobi Field orthogonal to $\gamma$ and so can be written in the form $V= f_1 E_1 +f_2 E_2$, where $(E_1, E_2)$ is an orthonormal basis of the (horizontal) plane and parallel along $\gamma$. We can suppose that $V(0)=E_1$ and $V(c)=E_1 \cos{(2\pi/3)} + E_2 \sin{(2\pi/3)}$. Now solving the Jacobi Field equation $V''+V=0$ we get $f_1(s)=\cos{(s)}+ \frac{\cos{(2\pi/3)-\cos{(c)}}}{\sin{(c)}}\sin{(s)}$ and $f_2(s)= \frac{\sin{(2\pi/3)}}{\sin{(c)}}\sin{(s)}$.\\

Finally  $$g_{hex}^c(V,\nabla_T V)|_0^c=f_2(c)f_2'(c) + f_1(c)f_1'(c)-f_1'(0)$$
$$= \sin^2(2\pi/3)(\cos(c)-\cos(2\pi/3))+\cos(2\pi/3)(\cos^2(c)-\cos^2(2\pi/3))$$

\end{proof}

\begin{remarque}
This lemma shows that there exists a neighbourhood $U$ of the geodesic $\gamma$ in which $\gamma$ is of minimum length among the geodesics joining any point $m$ to  $\varphi_c(m)$.

\end{remarque}

\begin{theo}
If the norm of the shortest vector of $\Delta$ is $\pi/3$,
the quotient metric on $C_3$ with respect to the metric $(T,g_{hex})$ (i.e. the manifold $(C_3,g_{hex})$)
has a systolic ratio bigger than $2/\sqrt{3}$, namely bigger
than the best systolic ratio for flat metrics on $C_3$.

\end{theo}

\begin{proof}

We consider in the neighborhood $U$ a hexagon $H$ "parallel" (c.f. remark \ref{parallel}) to the boundary $\partial D_p$. Let $\delta$ be a curve in  $(\RR^3,g_{hex})$ from a point $m$ in $D_p$ to $\varphi_c(m)$, the minimal projection of $\delta$ on the boundary $\partial H$ gives a curve $\delta '$ in $U$ joining two conjugate points by the transformation $\varphi_c$. Then we have by lemma \ref{miniproj} and remark \ref{parallel}

$$l(\delta) \geq l(\delta ') \geq c$$

The same arguments as those used in lemma \ref{distsigma} (section \ref{C2}) show that $d_{(T,g_{hex})}(m,\varphi_c(m)) \geq c$.\\

Now, passing to the limit when $c\rightarrow 2\pi/3$ we get 
$$ d_{(T,g_{hex})}(m,\varphi(m)) \geq 2\pi/3$$

This allows us to calculate the systole of $(C_3,g_{hex})$, when $a= \pi/6$: it is equal to $2\pi/3$. The volume is equal to 

$$\int_{0}^{\frac{2\pi}{3}} \iint_{D}  \cos \sqrt{x^2+y^2} dy dx dz $$

As before we calculate this integral using Maple, and we get an approximation of $\frac{Sys(C_3,g_{hex})^3}{Vol(C_3,g_{hex})} \simeq 1.24$.

\end{proof}

\begin{remarque}
The previous proof is also valid for the manifolds $(C_6,g_{hex})$ and $(C_4,g_c)$, and allows us to find a good lower bound for their systoles. But the method we used for these manifolds is a lot more simple (we just used the triangular inequality).

\end{remarque}

\section{Extremality of $(C_2,g_{hex})$ in its conformal class}\label{c2ext}

\begin{theo}\label{ext} The metric $g_{hex}$ on the manifold $C_2$ (see theorem \ref{c2sing}) is extremal in its conformal class.

\end{theo}

The proof of this theorem relies on the following result of C. Bavard \cite{bavard3}.\\

Let $(M,g)$ be a compact essential Riemannian manifold of dimension $n$ and let $\Gamma$ be the space of the systolic curves of $(M,g)$. For every Radon measure $\mu$ on $\Gamma$, we associate a measure $^* \mu $ on $M$ by setting for $\varphi \in C^0(M,\RR)$\\

$$ <^* \mu, \varphi>= <\mu,\overline{\varphi}>$$

where $\overline{\varphi}(\gamma)=\int{\varphi \circ \gamma (s) ds}$, $ds$ is the arc length of $\gamma$ with respect to $g$. Then we have

\begin{theo} (\cite{bavard2},\cite{bavard3} and \cite{jenk})

The Riemannian manifold $(M,g)$ is minimal in its conformal class if and only if there exists a positive measure $\mu$, of mass $1$, on $\Gamma$ such that

$$^*\mu=\frac{Sys(g)}{Vol(g)} \cdot dg$$

where $dg$ is the volume measure of $(M,g)$.

\end{theo}

\begin{proof}[Proof of theorem \ref{ext}]

A fundamental domain for the action of
$2\Delta\times 2\pi\ZZ$ is constituted of four hexagonal prisms
$D_p$. We consider in a domain $D_p$ the Riemannian metric $g_{hex}$
written in the form $g=dr^2+r^2d\theta^2 +\cos ^2 (r) dz^2$. Then we consider in the surface
$\theta=\varphi=constant$ (inside $D_p$) the geodesic ${\gamma_{(\varphi,0)}} ^a$ going
through the points $(0,0,-\pi/2)$, $(0,0,\pi/2)$ and $(a,0,0)$, where
$0 \leq a \leq f(\varphi)$ and $f$ is a $\pi/3$-periodic function
defined in $[0,\pi/3]$ by

$$ f(\varphi)= \left\lbrace \begin{array}{ll}
\frac{\pi}{4\cos \varphi} &  if\ \ 0\leq \varphi \leq \pi/6\\
\frac{\pi}{2(\cos \varphi +\sqrt{3}\sin \varphi)} & if\ \ \pi/6 \leq \varphi \leq \pi/3\\
\end{array} \right.$$

\begin{remarque} 

Inside a domain $D_p$, any surface of equation
$\theta=constant$ is a M\"obius band with boundary or a Klein bottle of constant
positive curvature. The maximum of the $r$-component is determined
by the function $f$.

\end{remarque}

Now, we consider the images ${\gamma_{\varphi,\phi}}^a$ of ${\gamma_{(\varphi,0)}} ^a$ by the
isometries $(r,\theta,z) \mapsto (r,\theta,z+\phi)$. We put on the space $\Gamma=\{{\gamma_{\varphi,\phi}}^a: \varphi \in \RR/\pi\ZZ$, $\phi \in \RR/ \pi \ZZ$, and $0\leq a \leq f(\varphi)\}$ the following mesure:\\
$$\mu=h(a,\varphi)da\otimes d \varphi \otimes d \phi$$

\begin{remarque} 

The function $h$ depends on two parameters
($a$ and $\varphi$) since the isometry group of $g_{hex}$ is
one dimensional.
\end{remarque}

We write ${\gamma_{\varphi,\phi}}^a(t)=(r(t,a),\varphi,\phi+t)$, where $r$ is a $C^1$ piecewise function.\\
Let $\psi$ be a continuous function on $(C_2,g_{hex})$, then\\

$$< ^* \mu, \psi>=\int_{a=0}^{f(\varphi)} \int_{\varphi=-\pi/2}^{\pi/2}\int_{\phi=-\pi/2}^{\pi/2}\int_{t=-\pi/2}^{\pi/2}{\psi(r(t,a),\varphi,\phi+t)\sqrt{\cos^2(r(t,a))+\Big (\frac{\partial r}{\partial t}(t,a)\Big )^2}h(a,\varphi)dt\ d\phi\ d\varphi \ da} $$

The changes of variables: $x=r (t,a)$, $\varphi=y$ and $z=\phi+t$ give

$$ < ^* \mu, \psi>=2 \int_{y=-\pi/2}^{\pi/2}\int_{z=-\pi/2}^{\pi/2}\int_{x=0}^{f(y)}\int_{a=x}^{f(y)}{\varphi(x,y,z)\sqrt{ \Big (\frac{\cos(y)}{\alpha(y,a)}\Big)^2+1 }\ h(a,y)dy\ dz\ dx\ da} $$

where we set $\alpha(x,a)=\frac{\partial r}{\partial t}(t,a)$.\\

Finally

$$^* \mu(r,\theta,z)= 2 \frac{\chi (r\leq f(\theta))}{r\cos(r)} \int_{a=r}^{f(\theta)}\sqrt{ \Big (\frac{\cos(r)}{\alpha(r,a)}\Big)^2+1}\ h(a,\theta)da \ dg_{hex}(r,\theta,z)$$

To calculate the function $\alpha(y,a)$, we will use the fact that
the curves ${\gamma_{\varphi,\phi}}^a$ are geodesics. Let
$g_{hex}^\varphi$ be the metric induced by $g_{hex}$ on the
hypersurface $\theta=\varphi$. Then $g_{hex}^\varphi=dr^2+\cos^2(r) dz^2$
can be written in the form $g(y)(dx^2+dy^2)$ and the geodesics $\{
{\gamma_{\varphi,\phi}}^a:\ 0 \leq a \leq f(\varphi), \phi \in
\RR/\pi\ZZ \}$ will have to satisfy the condition (see \cite{pu})
$$\frac{d}{dx}\Big ( y'\big ( \frac{g(x)}{1+{y'}^2} \big )^{1/2}\Big )=0 $$

where $y'=\frac{dy}{dx}$. In our case, we have $g(x)= \cos^2(r)$ and
$y'= \frac{\cos (r)}{\alpha(r,a)}$. We get 
$$\alpha^2(r,a)=\frac{\cos ^2(r)}{\cos ^2(a)}\big ( \cos ^2(r)-\cos ^2(a) \big )$$ 

Finally we have

$$^* \mu(r,\theta,z)= 2 \frac{\chi (r\leq f(\theta))}{r} \Big( \int_{a=r}^{f(\theta)}{\big ( \cos ^2(r)-\cos ^2 (a) \big )^{-\frac{1}{2}} h(a,\theta)da} \Big) \ dg_{hex}(r,\theta,z)$$

Now we are capable of calculating the function $h$. It should
satisfy the equation

$$^* \mu= dg_{hex}$$
on the band $r \leq f(\theta)$. Then we have

$$2\int_{a=r}^{f(\theta)}{\big ( \cos ^2(r)-\cos ^2 (a) \big )^{-\frac{1}{2}} h(a,\theta)da}=r$$

We put $z=\cos ^2(a)-\cos^2(f(\theta))$ and  $t=\cos^2 r-\cos^2(f(\theta))$,
we get the equation

$$\int_0^t{(t-z)^{-1/2}\beta(z,\theta)dz=\delta(t)} $$

where $\delta(t)=\cos^{-1}(\sqrt{t+\cos^2(f(\theta))})$. The solution of
this equation is

$$\beta(z)=\frac{1}{\pi}\Big ( \frac{\delta(0)}{\sqrt{z}} + \int_0^z (z-x)^{1/2}\delta '(x) dx \Big) $$

We can easily verify that the integral in the preceding expression
does converge and that

$$ \beta(z,\theta)=\frac{1}{\pi}\Big( \frac{f(\theta)}{\sqrt{t}}-\int_0^z\frac{dx}{2\big((z-x)(x+\cos ^2(f(\theta)))(1-x-\cos^2(f(\theta)))\big)^{1/2}} \Big)  $$

Finally
$$ h(a,\theta)=\frac{4\cos(a)sin(a)}{\pi} \Big( \frac{f(\theta)}{\sqrt{\cos ^2 a-\cos^2(f(\theta))}}- I(a,\theta)\Big )$$

where
$$  I(a,\theta)=\int_0^{\cos^2 a-\cos^2(f(\theta))} \frac{dx}{2\big((\cos ^2a-\cos^2(f(\theta))-x)(x+\cos ^2(f(\theta)))(1-x-\cos^2(f(\theta)))\big)^{1/2}}$$

We define the same way $\mu$ on the the family  of systolic
geodesics $s(\Gamma)$ where $s$ is the symmetry with respect to the
boundary of the hexagonal prism $D_p$.

Finally we have $^*\mu=dg_{hex}$, which proves that $g_{hex}$ is
extremal in its conformal class.

\end{proof}

\begin{remarque} The previous result still holds for metrics on $C_2$ obtained by taking the quotient of $(\RR^3,h)$ by the group generated by $2\Delta$ and $\sigma$ with no condition on the parameter $a$ (the length of the shortest vector in $\Delta$) but the calculations become a bit more complicated.

\end{remarque}

\section{Comparison between $(C_2,g_{hex})$ and the face-centered cubic torus} \label{torusC2}

Among flat tori of dimension $3$, the face-centered cubic one is the best for the isosystolic inequality. It is the quotient of $\RR^3$ by a face-centered cubic lattice. It is known that this torus, that we denote by $T_{fcc}$, is a very good candidate to realize the systolic constant of tori of dimension $3$. It satisfies the following properties: \\

\begin{itemize}

\item At any point in $T_{fcc}$ there exists exactly $6$ systolic geodesics going through the point. 
 
\item The systolic geodesics of any systolic class of $T_{fcc}$ cover the torus. A systolic class is an element of the fundamental group that contains at least one systolic geodesic.

\item It is extremal (for the isosystolic inequality) in its conformal class.

\end{itemize}

Our singular metric $(C_2,g_{hex})$ verifies the second and third properties, and a stronger one than the first: At any point outside the singularity of $(C_2,g_{hex})$, there exists infinitely many systolic geodesics going through the point.\\

For the points on the singularity, there are $4$ systolic geodesics going through any of these points: $3$ in the horizontal flat 2-torus and $1$ in the surface $\theta= constant$. The singularity subset anyway has zero measure.

We think that as for the face-centered cubic torus, the manifold $(C_2,g_{hex})$ is a very good candidate to realize the systolic constant because it has an abondance of systolic geodesics that can be seen by the fact that it satisfies the properties mentioned above.\\

When speaking about the metrics $(C_3,g_{hex})$,$(C_6,g_{hex})$ and $(C_4,g_c)$, they still satisfy the property of being covered by systolic geodesics mentioned in the introduction. But we cannot say if they satisfy something stronger as for the manifold $(C_2,g_{hex})$ since we do not have much information about the length of the vertical geodesics.\\ 

The following table gives a comparison between the greatest systolic ratio among the flat metrics on each one of the manifolds $C_2,C_3,C_4$ and $C_6$ ($\tau$(flat)) and the greatest systolic ratio among the singular metrics we have constructed on these manifolds ($\tau$(singular)).

\begin{table}[htbp]\renewcommand{\arraystretch}{1.2}
\begin{center}
\begin{tabular}[b]{|*{5}{c|}}
   \hline
type    &$\tau$(flat)&approximate value&$\tau$(singular)\\
   \hline
   $C_2$&$\frac{2}{\sqrt{3}}$&$\thickapprox  1,154$&$\thickapprox 1,38$\\
   \hline
   $C_3$&$\frac{2}{\sqrt{3}}$&$\thickapprox  1,154$&$\thickapprox 1,24$\\
   \hline
   $C_4$&$1$&$1$&$\thickapprox 1,05$\\
   \hline
   $C_6$&$\frac{2}{\sqrt{3}}$&$\thickapprox  1,154$&$\thickapprox 1,18$\\
   \hline
   \end{tabular}\end{center}
\end{table}

\emph{Acknowledgements:} The first author is grateful to Benoit Michel for useful discussions and remarks  especially on \emph{lemma 5}. He also thanks the "Lebanese Association for Scientific Research" for hospitality and support. 
Both authors thank the referee for constructive criticism.

\end{document}